\documentclass[12pt]{amsart}
\usepackage{amsmath,amssymb,amsthm,mathabx} 
\usepackage[all]{xy}
\CompileMatrices
\usepackage{verbatim}
\usepackage{comment}
\usepackage{float}
\restylefloat{table}
\usepackage{caption}
\usepackage{pgfplots}
\usepackage{framed}
\usepackage{xcolor}
\usepackage{enumerate}
\usepackage{mathtools}
\usepackage{hyperref}
\usepackage[nameinlink]{cleveref}

\usepackage{pst-poker}
\title{The 21 Card Trick and Its Generalization}
\author{Dibyajyoti Deb}
\address{Department of Mathematics, Oregon Institute of Technology, Klamath Falls, OR, USA}
\email{dibyajyoti.deb@oit.edu}

\theoremstyle{plain}
\newtheorem{theorem}{Theorem}[section]
\newtheorem{lemma}[theorem]{Lemma}
\newtheorem{corollary}[theorem]{Corollary}

\newtheorem{claim}[theorem]{Claim}
\theoremstyle{definition}
\newtheorem{definition}[theorem]{Definition}

\newcommand{\mZ}{\mathbb{Z}}
\newcommand{\mR}{\mathbb{R}}

\DeclareMathOperator{\lcm}{lcm}

\newcommand\numberthis{\refstepcounter{equation}\tag{\theequation}}

\DeclarePairedDelimiter{\ceil}{\lceil}{\rceil}
\DeclarePairedDelimiter{\floor}{\lfloor}{\rfloor}

\begin{document} 

\begin{abstract}
The 21 card trick is well known. It was recently shown in an episode of the popular YouTube channel \href{https://www.youtube.com/watch?v=d7dg7gVDWyg}{Numberphile}. In that trick, the audience is asked to remember a card, and through a series of steps, the magician is able to find the card. In this article, we look into the mathematics behind the trick, and look at a complete generalization of the trick. We show that this trick can be performed with any number of cards. 
\end{abstract}

\maketitle

\section{Introduction to the 21 card trick} 
The \textbf{21 card trick} (21CT) is a very popular card trick. It was also recently shown in an episode of the popular YouTube channel \href{https://www.youtube.com/watch?v=d7dg7gVDWyg}{Numberphile}. We first explain how this trick is performed in a series of steps. For the purpose of this demonstration, we will assume that a magician (henceforth, Magi), is showing this trick to his friend and audience (henceforth, Audy).
\begin{enumerate}
\item Audy randomly chooses 21 cards from a deck of cards. He remembers one card from that set, shuffles the set of 21 cards and hands it back to Magi.
\item Magi requests Audy to pay attention as he puts the cards face up one at a time adjacent to each other creating 3 stacks of 7 cards each in the process. Magi asks Audy to tell him which stack contains his card.
\item Magi then puts the stack that contained Audy's card between the two other stacks, and then repeats Step (2) two more times. 
\item Magi now puts the cards in the table one at a time, and stops at the 11th card which turns out to be Audy's card.
\end{enumerate} 
Note that this trick can be made more ``magical" with some extra activities, which we will not discuss as they don't contribute anything to the main problem. Our current goal is to find out why the 11th card on the deck happened to be Audy's card. We will explore the mathematics behind this after we look at a detailed example of the trick. 

Once Audy hands the shuffled deck of cards to Magi, we know that Audy's card could be any one of the 21 cards in that deck. Let's number the cards from top to bottom of the deck by $1, 2, 3, \ldots, 21$. We shall use the term \textbf{\emph{deck id}} to denote the position of Audy's card in the deck. Similarly, we will use the term \textbf{\emph{stack id}} to denote the position of Audy's card in a stack from the top.
\begin{framed}
\begin{definition}
An \textbf{iteration} is considered to be the process of splitting the cards of a deck into stacks.
\end{definition}
\begin{definition}
The \textbf{deck id} of Audy's card after $k(\geq 1)$ iterations of splitting into stacks, denoted by $d_k$, is the position of his card from the top, in a deck of $n$ distinct cards. 
\end{definition}
\begin{definition}
The \textbf{stack id} of Audy's card after $k( \geq 1)$ iterations of splitting into stacks, denoted by $s_k$, is the position of his card from the top, in its  individual stack.
\end{definition}
\end{framed}

\noindent
Note that the initial deck id of Audy's card before any iterations happen is denoted by $d_0$. We now show a detailed example of the 21CT in action.

\begin{enumerate}
\item Audy selects a random collection of 21 cards. The cards are:

\smallskip
\noindent
\psset{inline=boxed}
\fourc, \tenc, \tred, \tenh, \tend, \eigd, \sixc, \tres, \tens, \twoc, \fives, \sixh, \twod, \eigc, \Ac, \Kd, \ninec, \fours, \nineh, \Kh, \Jd

\smallskip
\noindent
He remembers a card from this set. Suppose it's \fours. He then shuffles this deck again and hands it over to Magi. The shuffled deck looks like this:

\smallskip
\noindent
\eigd, \tend, \nineh, \twoc, \tred, \fives, \fourc, \sixc, \sixh, \twod, \tenc, \tenh, \Kd, \Jd, \ninec, \eigc, \tres, \Ac, \tens, \fours, \Kh \, \, ($d_0 = 20$)

\enspace
\noindent
Note that an iteration hasn't happened yet. Thus, $s_0$ is not defined, and $d_0 = 20$ (as \fours \, is the 20th card in the deck at the moment).

\smallskip
\item Now Magi does the first iteration by putting these cards face up one at a time adjacent to each other creating 3 stacks. The stacks then look like this:

\smallskip
\noindent
\begin{table}[H]
\centering
\begin{tabular}{|c|c|c|}
\hline
Stack 1 & Stack 2 & Stack 3 \\
\hline 
\eigd & \tend & \nineh \\ 
\hline 
\twoc & \tred & \fives \\ 
\hline 
\fourc & \sixc & \sixh \\ 
\hline 
\twod & \tenc & \tenh \\ 
\hline 
\Kd & \Jd & \ninec \\ 
\hline 
\eigc & \tres & \Ac \\ 
\hline 
\tens & \fours & \Kh \\ 
\hline 
\end{tabular} 
\caption{Stacks after Iteration 1 ($s_1 = 7$ as \protect\psset{inline=boxed} \protect\fours \, is the 7th card in the second stack)}
\end{table}

\smallskip
\noindent
He then asks Audy which stack contains his card. Audy replies by saying that his card is in Stack 2 (\fours \, is in Stack 2). 

\item Magi then puts Stack 2 in between Stack 1 and Stack 3. The set of 21 cards now look as follows from top to bottom: 

\smallskip
\noindent
\eigd, \twoc, \fourc, \twod, \Kd, \eigc, \tens, \tend, \tred, \sixc, \tenc, \Jd, \tres, \fours, \nineh, \fives, \sixh, \tenh, \ninec, \Ac, \Kh \, \, ($d_1 = 14$)

\smallskip
\noindent 
Magi now repeats Step (2) two more times. The stacks after iteration 2 looks like this:

\smallskip
\noindent
\begin{table}[H]
\centering
\begin{tabular}{|c|c|c|}
\hline
Stack 1 & Stack 2 & Stack 3 \\
\hline 
\eigd & \twoc  & \fourc  \\ 
\hline 
\twod & \Kd  & \eigc  \\ 
\hline 
\tens & \tend  & \tred  \\ 
\hline 
\sixc & \tenc & \Jd   \\ 
\hline 
\tres  & \fours & \nineh  \\ 
\hline 
\fives & \sixh & \tenh  \\ 
\hline 
\ninec & \Ac & \Kh \\ 
\hline 
\end{tabular} 
\caption{Stacks after Iteration 2 ($s_2 = 5$)}
\end{table}

\noindent
\psset{inline=boxed}
He then again asks Audy which stack contains his card. Audy replies by saying that his card is in Stack 2 again (\fours \, is in Stack 2). Magi again puts Stack 2 in between Stack 1 and Stack 3. The set of 21 cards now look as follows from top to bottom:

\smallskip
\noindent
\eigd, \twod, \tens, \sixc, \tres, \fives, \ninec, \twoc, \Kd, \tend, \tenc, \fours, \sixh, \Ac, \fourc, \eigc, \tred, \Jd, \nineh, \tenh, \Kh \, \, ($d_2 = 12$)

\smallskip
\noindent
Magi again repeats the process of splitting the cards into 3 stacks for the third and final time. The stacks after iteration 3 look like this:

\smallskip
\noindent
\begin{table}[H]
\centering
\begin{tabular}{|c|c|c|}
\hline
Stack 1 & Stack 2 & Stack 3 \\
\hline 
\eigd & \twod  & \tens  \\ 
\hline 
\sixc & \tres  & \fives  \\ 
\hline 
\ninec & \twoc  & \Kd  \\ 
\hline 
\tend & \tenc & \fours   \\ 
\hline 
\sixh  & \Ac & \fourc  \\ 
\hline 
\eigc & \tred & \Jd  \\ 
\hline 
\nineh & \tenh & \Kh \\ 
\hline 
\end{tabular} 
\caption{Stacks after Iteration 3 ($s_3 = 4$)}
\end{table}

\smallskip
\noindent
Now for the final time, Magi asks Audy about the stack that contains his card. Audy replies by saying that his card is in Stack 3 (\fours \, is in Stack 3). Magi then puts Stack 3 in between Stack 1 and Stack 2. The set of 21 cards now look as follows from top to bottom:

\smallskip
\noindent
\eigd, \sixc, \ninec, \tend, \sixh, \eigc, \nineh, \tens, \fives, \Kd, \fours, \fourc, \Jd, \Kh, \twod, \tres, \twoc, \tenc, \Ac, \tred, \tenh \, \, ($d_3 = 11$)

\item Magi now puts each card from the top face down on the table one by one, and flips over the \textbf{11th card} which turns out to be Audy's card, \fours.

\smallskip
\noindent
\begin{center}
\eigd, \sixc, \ninec, \tend, \sixh, \eigc, \nineh, \tens, \fives, \Kd, 
\fcolorbox{blue}{green}{\fours}
\end{center}

\end{enumerate}


\smallskip
\noindent
Magi correctly finds out Audy's card, and it leaves Audy startled. He wonders what kind of voodoo did Magi apply in all of this. Little does he know about the power of mathematics behind this trick.

\section{The Mathematics behind the 21 card trick}
Most card tricks rely on mathematics. The 21CT is no exception. We now look at the mathematics behind this trick. 

\smallskip
\noindent
Before we look into the steps of the 21CT again, we would like to look at the definition of the \textbf{ceiling} and \textbf{floor} functions, and three lemmas that will be used extensively in this article. 

\bigskip
\noindent

\begin{framed}
\begin{definition}
Suppose $x \in \mR$. The \textbf{ceiling} of $x$ denoted by $\ceil{x}$ is the smallest integer greater than or equal to $x$. In general, if $\ceil{x} = n \in \mZ$, then $n-1 < x \leq n$.
\end{definition}
\begin{definition}
Suppose $x \in \mR$. The \textbf{floor} of $x$ denoted by $\floor{x}$ is the largest integer less than or equal to $x$. In general, if $\floor{x} = n \in \mZ$, then $n \leq x < n+1$. In this case $x - n = x - \floor{x}$ is called the \emph{fractional part} of $x$, and is denoted by $\{x\}$.
\end{definition}
\end{framed}

\begin{lemma}\label{lem1}
For $x, y \in \mR$, if $x \leq y$, then
$$\ceil{x} \leq \ceil{y}$$
\end{lemma}

\begin{proof}
We divide this into two cases.

\enspace

\noindent
\underline{Case 1 :} ($y \leq \ceil{x}$) As $\ceil{x} - 1 < x \leq \ceil{x}$ by definition, and $x \leq y$, therefore, $\ceil{x} - 1 < y \leq \ceil{x}$. This implies that $\ceil{y} = \ceil{x}$, and therefore, $\ceil{x} \leq \ceil{y}$.

\enspace

\noindent
\underline{Case 2 :} ($y > \ceil{x}$). By definition, $y \leq \ceil{y}$. Therefore, we have $\ceil{x} < y \leq \ceil{y}$. Hence, $\ceil{x} < \ceil{y}$.

\enspace

\noindent
Combining the two above case, we have the desired inequality.
\end{proof}

\begin{lemma}\label{lem2}
For $n \in \mZ$ and $x \in \mR$, 
$$\ceil{n+x} = n + \ceil{x}$$
\end{lemma}

\begin{proof}
Suppose $\ceil{n+x} = d$. Thus,
\begin{equation*}
\begin{split}
d-1  < n + & x \leq d \\
(d-n)-1 < \, & x \leq d-n
\end{split}
\end{equation*}
This implies, $\ceil{x} = d - n$. Adding $n$ to both sides, we have $n + \ceil{x} = d$ which shows the desired result.
\end{proof}

\begin{lemma}\label{lem3}
For $n, m \in \mZ$ with $m > 0$, and $x \in \mR$, 
$$\ceil[\Bigg]{\frac{n + \ceil{x}}{m}} = \ceil[\Bigg]{\frac{n+x}{m}}$$ 
\end{lemma}

\begin{proof}
Suppose $\ceil[\Bigg]{\dfrac{n+x}{m}} = d$. Thus,
\begin{alignat*}{2}
d-1 &< &\frac{n+x}{m} &\leq d \\
md - m &< &n + x &\leq md \\
md - m - n &< &x  \, \, \, \, \, &\leq md - n 
\end{alignat*}
As $m, d, n \in \mZ$, thus $md-m-n$ and $md-n$ are integers. By definition of the ceiling function, $x \leq \ceil{x}$. Hence, $md - m - n < \ceil{x}$. On the other hand, $x, md-n \in \mR$. Hence, by Lemma \ref{lem1}, $\ceil{x} \leq \ceil{md-n} = md-n$ as $md - n \in \mZ$. Thus, we have 

\begin{alignat*}{2}
md-m-n &< &\ceil{x} \, \, \, \, \, \, \, &\leq md-n \\
md - m &< &n + \ceil{x} &\leq md \\
\frac{md - m}{m} &< &\frac{n + \ceil{x}}{m} &\leq \frac{md}{m} \\
d-1 &< &\frac{n + \ceil{x}}{m} &\leq d
\end{alignat*}
This implies, $\ceil[\Bigg]{\dfrac{n + \ceil{x}}{m}} = d$, which proves the result.

\end{proof}

\noindent
Now we look back at the steps of the card trick again. 

\enspace

\begin{enumerate}[Step 1:]
\item Audy hands the shuffled deck to Magi. We denote the initial deck id of Audy's card by $d_0$ as no iterations of splitting into stacks has been done yet. Therefore, 
$$1 \leq d_0 \leq 21$$
\item Magi now puts the cards from the top, face up, on the table adjacent to each other, creating 3 stacks of 7 cards each. Audy's card is in one of the three stacks. We claim the following.
\begin{claim}\label{claim1}
If $s_k$ denotes the stack id, and $d_k$ denote the deck id of Audy's card after $k$ iterations, then,
$$s_{k} = \ceil[\Big]{\dfrac{d_{k-1}}{3}} \hspace{0.2in} \text{for} \hspace{0.1in} k \geq 1$$
\end{claim}
\begin{proof}
It is imperative for us to note that a stack id for Audy's card is created after every iteration. Hence, $s_k$ is defined only when $k \geq 1$. On the other hand, a deck id for Audy's card is created when the 21 cards are all in a single deck, which happens for the first time when Audy hands over the shuffled deck to Magi. Thus, $d_k$ is defined for $k \geq 0$. Also, note that the $k$th iteration which creates $s_{k}$ is performed using the 21 card deck after the $(k-1)$th iteration which creates $d_{k-1}$. Hence the relationship between $s_{k}$ and $d_{k-1}$. It is clear that the stack id of a card is the ``row number" of the card in the stack. Thus for $k \geq 1$, the first row consists of cards with $d_{k-1} = 1, 2, 3$. The second row consists of cards with $d_{k-1} = 4, 5, 6$ and so on. This means that the $n$th row consists of cards with $d_{k-1} = 3n-2, 3n-1, 3n$. As $n =s_{k}$. Thus,
\begin{alignat*}{2}
3s_{k}-3 &< &d_{k-1} &\leq 3s_{k} \\
s_{k} -1 &< &\frac{d_{k-1}}{3} &\leq s_{k}
\end{alignat*}
Therefore, $\ceil[\Big]{\dfrac{d_{k-1}}{3}} = s_{k}$.
\end{proof}
 
\noindent
As $1 \leq d_0 \leq 21$, hence, 
\begin{alignat*}{2}
\dfrac{1}{3} &\leq &\dfrac{d_0}{3} \, &\leq 7 \\
1 &\leq &\ceil[\Big]{\dfrac{d_0}{3}} &\leq 7 \qquad (\text{By Lemma } \ref{lem1}) \\
1 &\leq &s_1 \, \, &\leq 7 \qquad (\text{By Claim } \ref{claim1})
\end{alignat*}
 
\noindent
Now Magi asks Audy to tell him the stack that contains his card. Audy responds by saying that it's Stack 2.
 
\item Magi then puts Stack 2 in between the other two stacks and creates a full deck of 21 cards. Note that there is a stack of 7 cards on top of Stack 2 at the moment. The position of Audy's card in Stack 2 is currently $s_1$. Hence, the new deck id of Audy's card after the first iteration is, 
$$d_1 = 7+ s_1$$
Since $1 \leq s_1 \leq 7$, therefore,
\begin{alignat*}{2}
8 &\leq &7+s_1 &\leq 14 \\
8 &\leq &d_1 \, \, \, &\leq 14 
\end{alignat*}

\noindent
Now the second iteration is performed, and the stack id $s_2$ is created,
\begin{alignat*}{2}
8 &\leq &d_1  &\leq 14 \\
\dfrac{8}{3} &\leq &\dfrac{d_1}{3}  &\leq \dfrac{14}{3} \\
3 &\leq &\ceil[\Big]{\dfrac{d_1}{3}} &\leq 5 \qquad (\text{By Lemma } \ref{lem1}) \\
3 &\leq &s_2 \, \, &\leq 5 \qquad (\text{By Claim } \ref{claim1})
\end{alignat*}

\noindent
Note that Magi now knows that the stack id of Audy's card is either 3, 4 or 5, but that is not good enough, as he needs to find the exact card. He again asks Audy to tell him the stack where his card belongs. Audy mentions that it's Stack 2, and Magi puts Stack 2 in between the other two stacks to create a deck of 21 cards. The position of Audy's card in Stack 2 is currently $s_2$. Therefore, the new deck id of Audy's card after the second iteration is,
$$d_2 = 7 + s_2$$
Now, $3 \leq s_2 \leq 5$, therefore,
\begin{alignat*}{2}
10 &\leq &7+s_2 &\leq 12 \\
10 &\leq &d_2 \, \, \, &\leq 12 
\end{alignat*}

\noindent
Now the third iteration is performed, and the stack id $s_3$ is created,
\begin{alignat*}{2}
10 &\leq &d_2 \, \, &\leq 12 \\
\dfrac{10}{3} &\leq &\dfrac{d_2}{3} \, \, &\leq \dfrac{12}{3} \\
4 &\leq &\ceil[\Big]{\dfrac{d_2}{3}} &\leq 4 \qquad (\text{By Lemma } \ref{lem1}) \\
4 &\leq &s_3 \, \, &\leq 4 \qquad (\text{By Claim } \ref{claim1})
\end{alignat*}

\noindent
This means $s_3 = 4$. Magi can now see the finish line. Magi has now found the exact stack id of Audy's card in a stack, but he does not know exactly which stack contains Audy's card. He thus asks Audy one final time about the stack that contains his card. Audy mentions Stack 3, and Magi puts Stack 3 in between the other two stacks.  The position of Audy's card in the middle stack is currently $s_3 = 4$. Hence, the new deck id of Audy's card after the third iteration is,
$$d_3 = 7 + s_3 = 11$$

\item Magi now knows that Audy's card is the 11th card from the top of the 21 card deck. He puts each of the card from the top on the table, face down, until he flips over the 11th card to Audy's delight.

\end{enumerate}

\section{Generalization of the 21 card trick}\label{sec3}
Now that we have seen how the 21CT works, and the mathematics behind it, we ask ourselves whether the same trick can be performed by Magi using a random set of, let's say $C$ number of cards. He also wants to find out the number of stacks he needs to split his cards into, where should he put the stack that contains Audy's card, how many iterations should he perform, and lastly, the deck id of Audy's card after the final iteration.

\noindent
Let's help out Magi perform his trick. Before we proceed, we introduce some notations.
\begin{itemize}
\item Number of given cards = $C(>0) \in \mZ$.
\item Number of stacks to split into = $n(>0 \text{ and} \leq C) \in \mZ$.
\item Number of stacks to put on top of the stack which contains Audy's card = $j(\geq 0 \text{ and} <n)\in \mZ$.
\item Number of iterations to be performed = $k(>0) \in \mZ$.
\item Deck id of Audy's card after the final iteration = $l(>0) \in \mZ$.
\end{itemize}
We have thus created a 4-tuple ($C, n, j, k$) that provides information about the magic trick. 
\begin{framed}
\begin{definition}
A magic trick ($C, n, j, k$) is \textbf{solvable}, if there exists integers $k$ and $l$ such that the magic trick can be performed with the given parameters in $k$ iterations. In this case, we write $$(C, n, j, k) = l$$
If such a $k$ or $l$ does not exist then the magic trick cannot be performed, and we say that ($C, n, j, k$) is \textbf{not solvable}.
\end{definition}
\end{framed}

\noindent
The 21 card trick is represented by (21, 3, 1, 3) and is solvable and $(21, 3, 1, 3) = 11$. Following along the lines of the 21CT, our main assumption here is that each stack contains the same number of cards. Thus, $n \mid C$. Thus, $n \leq C$.

\noindent
Before we proceed any further, we first look at a generalization of Claim \ref{claim1} in the form of a theorem and an additional result.

\begin{theorem}\label{thm4}
Suppose $C$ cards are split into $n$ stacks during an iteration. If $s_k$ denotes the stack id, and $d_k$ denote the deck id of Audy's card after $k$ iterations, then,
$$s_{k} = \ceil[\Big]{\dfrac{d_{k-1}}{n}} \hspace{0.2in} \text{for} \hspace{0.1in} k \geq 1$$
\end{theorem}

\begin{proof}
The proof of this theorem is very similar to the earlier proof of Claim \ref{claim1}. We know that the stack id of a card is the ``row number" of the card in the stack. Thus for $k \geq 1$, the first row consists of cards with $d_{k-1} = 1, 2, 3, \ldots, n$. The second row consists of cards with $d_{k-1} = n+1, n+2, \ldots, 2n$, and so on. This means that the $q$th row consists of cards with $d_{k-1} = (q-1)n + 1, (q-1)n + 2, \ldots, (q-1)n + n$. As $q = s_{k}$. Hence, 
\begin{alignat*}{2}
n(s_{k}-1) &< &d_{k-1} &\leq n(s_{k} - 1) + n \\
n(s_{k}-1) &< &d_{k-1} &\leq ns_{k} \\
s_{k}-1 &< &\frac{d_{k-1}}{n} &\leq s_{k}
\end{alignat*}
Therefore, $\ceil[\Big]{\dfrac{d_{k-1}}{n}} = s_{k}$.
\end{proof}

\begin{theorem}\label{thm5}
Suppose $k \geq 1$, then
$$d_k = \Big(\frac{C}{n}\Big)j + s_k$$
\end{theorem}

\begin{proof}
As $C$ cards are being split into $n$ stacks, hence each stack contains $\dfrac{C}{n}$ cards. $s_k$ is the stack id of the desired card after $k$ iterations. $j$ stacks are kept on top of the stack that contains the desired card. There are a total of $\Big(\dfrac{C}{n}\Big)j$ cards in these $j$ stacks. Thus, the new deck id after $k$ iteration is,
$$d_k = \Big(\frac{C}{n}\Big)j +s_k \hspace{0.2in} \text{for} \hspace{0.1in} k \geq 1$$
\end{proof}

\noindent
Note that our goal is to find an exact value for $d_k$. This would tell us the deck id of the desired card which would successfully conclude the trick. Due to Theorem \ref{thm5}, we now have an expression for $d_k$ that we can investigate further. This results in the next theorem which will be used extensively in the next section.

\begin{theorem}\label{thm6}
Suppose $\dfrac{C}{n} = m$ and $k \geq 1$. Then,
$$d_k = 
\begin{cases}
mj + \ceil[\Bigg]{\dfrac{mjn\Big(\dfrac{n^{k-1}-1}{n-1}\Big) + d_0}{n^k}}, & \text{for} \hspace{0.1in} n > 1 \\
d_0, & \text{for} \hspace{0.1in} n = 1
\end{cases}
$$
\end{theorem}

\begin{proof}
\underline{Case 1 :} ($n>1$)
We will use mathematical induction on $k$ to prove this result. 

\noindent
For $k=1$, L.H.S. = $d_1 = mj + s_1$ from Theorem \ref{thm5}. Thus, $d_1 = mj + \ceil[\Big]{\dfrac{d_0}{n}}$.

\noindent
R.H.S. = $mj + \ceil[\Bigg]{\dfrac{mjn\Big(\dfrac{n^{0}-1}{n-1}\Big) + d_0}{n}} = mj + \ceil[\Big]{\dfrac{d_0}{n}}$. Hence, the result is true for $k=1$. Now assume the result is true for some $k = t \geq 1$. Thus,
$$d_t = mj + \ceil[\Bigg]{\frac{mjn\Big(\dfrac{n^{t-1}-1}{n-1}\Big) + d_0}{n^t}}$$
Now,
$$d_{t+1} = mj + s_{t+1} = mj + \ceil[\Big]{\dfrac{d_t}{n}} = mj + \ceil[\Bigg]{\dfrac{mj + \ceil[\Big]{\frac{mjn\Big(\frac{n^{t-1}-1}{n-1}\Big) + d_0}{n^t}}}{n}}$$
As $mj$ and $n$ are integers, hence by Lemma \ref{lem3}, 
$$\ceil[\Bigg]{\dfrac{mj + \ceil[\Big]{\frac{mjn\Big(\frac{n^{t-1}-1}{n-1}\Big) + d_0}{n^t}}}{n}} = \ceil[\Bigg]{\dfrac{mj + \frac{mjn\Big(\frac{n^{t-1}-1}{n-1}\Big) + d_0}{n^t}}{n}} = \ceil[\Bigg]{\dfrac{mjn^t + \Big(\frac{mjn^t-mjn}{n-1}\Big) +d_0}{n^{t+1}}}$$
Now,
$$\ceil[\Bigg]{\dfrac{mjn^t + \Big(\dfrac{mjn^t-mjn}{n-1}\Big) +d_0}{n^{t+1}}} = \ceil[\Bigg]{\dfrac{\Big(\dfrac{mjn^{t+1}-mjn}{n-1}\Big) +d_0}{n^{t+1}}} = \ceil[\Bigg]{\dfrac{mjn\Big(\dfrac{n^{t}-1}{n-1}\Big) +d_0}{n^{t+1}}}$$
Therefore,
$$d_{t+1} = mj + \ceil[\Bigg]{\dfrac{mjn\Big(\dfrac{n^{(t+1)-1}-1}{n-1}\Big) +d_0}{n^{(t+1)}}}$$
and thus the result is true for $k = t+1$. Therefore, by induction, the result is true for all $k \geq 1$.

\noindent
\underline{Case 2 :} ($n=1$) As $n=1$, and $0 \leq j < n$, hence $j = 0$. By Theorem \ref{thm5}, 
$$d_k = s_k \qquad for \hspace{0.1in} k \geq 1$$
Now, according to Theorem \ref{thm4}, $s_k = \ceil{d_{k-1}} = d_{k-1}$ for $k \geq 1$. Therefore,
$$d_k = d_{k-1} \qquad for \hspace{0.1in} k \geq 1$$
Solving this recurrence relation, we get $d_k = d_0$.
\end{proof}

\section{Main Result}
\noindent
The 21CT and the results that we have seen so far pushes us to the more general question. 
\begin{framed}
\noindent
\begin{center}\label{gen-ques}
\underline{\textbf{General Question}}
\end{center}
Given integers $C, n,$ and $j$, does there exist a positive integer $k$, such that the trick $(C, n, j, k)$ is solvable? In that case, what is $(C, n, j, k)$?
\end{framed}

\noindent
We answer this question using an algorithm. We will then look at the proof of the results in our algorithm.

\begin{framed}\label{algorithm}
\begin{enumerate}
\item Start with given integers $C, n, j$.
\item Find $m = \dfrac{C}{n}$ and $b = \dfrac{mj}{n-1}$ (for $n>1$). 
\item If $C = 1$, then $(1, 1, 0, k)$ is solvable for any $k \geq 1$. In this case, $(1, 1, 0, k) = 1$. 
\item If $C>1$ and $n = 1$, then $(C, 1, 0, k)$ is not solvable for any $k \geq 1$.
\item If $C, n > 1$, then $(C, n, 0, k)$ is solvable for any integer $k \geq \log_nC$. In this case, $(C, n, 0, k) = 1$.
\item If $C, n > 1$, then $(C, n, n-1, k)$ is solvable for any integer $k > \log_n(C-1)$. In this case, $(C, n, n-1, k) = C$.
\item If $C, n > 1$ and $0<j<n-1$, then $(C, n, j, k)$ is solvable if, and only if, $(n-1) \not\mid mj$. In this case, $(C, n, j, k)$ is solvable for any integer $k > \log_nt$, where $t = \max\Big\{\dfrac{C-bn}{1-\{b\}}, \dfrac{bn-1}{\{b\}}\Big\}$ and, $(C, n , j, k) = mj + \floor{b} + 1$. 
\end{enumerate}
\end{framed}

\noindent
The steps (3) - (7) of the above algorithm gives a complete answer to the general question posed at the beginning of this section. We will prove each step in the form of a theorem, culminating with proof of the all important Step (7).

\begin{theorem} (Step (3)) \label{thm:step3}
If $C = 1$, then $(1, 1, 0, k)$ is solvable for any $k \geq 1$ with $(1, 1, 0, k) = 1$.
\end{theorem}

\begin{proof}
This is a trivial case as there is only 1 card. As $0 < n \leq C$, therefore, $n=1$. Also, as $j < n$, hence $j = 0$. By Theorem \ref{thm6}, $d_k = d_0$. We know that the initial deck id $d_0$ of the desired card is between $1$ and $C$. Thus, $1 \leq d_0 \leq 1$. Therefore, $d_k = d_0 = 1$. Thus $(1, 1, 0, k)$ is solvable for any $k \geq 1$, with $(1, 1, 0, k) = 1$.
\end{proof}

\begin{theorem} (Step (4)) \label{thm:step4}
If $C>1$ and $n = 1$, then $(C, 1, 0, k)$ is not solvable for any $k \geq 1$.
\end{theorem}

\begin{proof}
By Theorem \ref{thm6}, $d_k = d_0$, as $n=1$. But $1 \leq d_0 \leq C$. Thus,
$$1 \leq d_k \leq C \hspace{0.3in} \text{ for } \hspace{0.1in} k \geq 1$$
As $C>1$, there is no specific value of $d_k$ that we can find for any $k$. Thus, the trick $(C, 1, 0, k)$ is not solvable.
\end{proof}

\begin{theorem} (Step (5)) \label{thm:step5}
If $C, n > 1$, then $(C, n, 0, k)$ is solvable for any integer $k \geq \log_nC$. In this case, $(C, n, 0, k) = 1$. 
\end{theorem}

\begin{proof}
From Theorem \ref{thm6} with $j = 0$, we have 
$$d_k = \ceil[\Big]{\frac{d_0}{n_k}}$$
As $1 \leq d_0 \leq C$, therefore, 
\begin{alignat*}{2}
\dfrac{1}{n^k} &\leq & \dfrac{d_0}{n^k} \, \,&\leq \dfrac{C}{n^k} \\
\ceil[\Big]{\dfrac{1}{n^k}} &\leq & \, \ceil[\Big]{\dfrac{d_0}{n^k}} &\leq \ceil[\Big]{\dfrac{C}{n^k}} 
\end{alignat*}
As $n>1$, therefore $0 < \dfrac{1}{n^k} < 1$. Hence,
\begin{alignat*}{2}
1 &\leq & \, d_k & \leq \ceil[\Big]{\dfrac{C}{n^k}} \numberthis \label{step5:ineq1}
\end{alignat*}

\noindent
Now, $k \geq \log_nC$ and $C>1$. This implies, 
\begin{align*}
n^k &\geq C \\
0 < & \frac{C}{n^k}  \leq 1
\end{align*}
Therefore, $\ceil[\Big]{\dfrac{C}{n^k}} = 1$. Using this result in inequality \eqref{step5:ineq1}, we have
$$1 \leq d_k \leq 1$$
Thus, $d_k = 1$. This shows that after $k$ iterations, we have a specific value for the the deck id $d_k$, and this value is the position of the desired card which happens to be the first card in the deck. Hence, the trick $(C, n, 0, k)$ is solvable for any integer $k \geq \log_nC$, and $(C, n, 0, k) = 1$. 
\end{proof}

\begin{theorem} (Step (6)) \label{thm:step6}
If $C, n > 1$, then $(C, n, n-1, k)$ is solvable for any integer $k > \log_n(C-1)$. In this case, $(C, n, n-1, k) = C$. 
\end{theorem}

\begin{proof}
Here $j = n-1$. Therefore, by Theorem \ref{thm6},
\begin{equation*}
\begin{split}
d_k = m(n-1) + \ceil[\Bigg]{\frac{m(n-1)n\Big(\dfrac{n^{k-1}-1}{n-1}\Big) + d_0}{n^k}} & = m(n-1) + \ceil[\Bigg]{\frac{m(n^k-n)}{n^k} + \frac{d_0}{n^k}} \\
& = m(n-1) + \ceil[\Bigg]{m + \frac{d_0 - mn}{n^k}}
\end{split}
\end{equation*}
As $m \in \mZ$, therefore, by Lemma \ref{lem2}, 
\begin{equation*}
\begin{split}
d_k = m(n-1) + \ceil[\Bigg]{m + \frac{d_0 - mn}{n^k}} &= m(n-1) + m + \ceil[\Bigg]{\frac{d_0 - mn}{n^k}} \\
&= mn + \ceil[\Bigg]{\frac{d_0 - C}{n^k}} \hspace{0.2in} \text{as } C = mn.
\end{split}
\end{equation*}
$d_0$ is the initial deck id before any iteration. Hence $1 \leq d_0 \leq C$. This implies,
\begin{alignat*}{2}
1-C &\leq & \, d_0 - C \,\,\, &\leq 0  \\
\ceil[\Bigg]{\frac{1-C}{n^k}} &\leq & \, \ceil[\Bigg]{\frac{d_0 - C}{n^k}} &\leq 0 \numberthis \label{step6:ineq1}
\end{alignat*}

\noindent
We are told that $k > \log_n(C-1)$, hence $n^k >C-1$. This implies, 
\begin{equation*}
\begin{split}
\frac{1-C}{n^k} &> -1 \\
\ceil[\Bigg]{\frac{1-C}{n^k}} & > -1
\end{split}
\end{equation*}
This along with inequality $\ceil[\Bigg]{\dfrac{1-C}{n^k}} \leq 0$ from \eqref{step6:ineq1}, implies that $\ceil[\Bigg]{\dfrac{1-C}{n^k}} = 0$. Thus,
$$0 \leq \ceil[\Bigg]{\frac{d_0 - C}{n^k}} \leq 0$$
which means $\ceil[\Bigg]{\dfrac{d_0 - C}{n^k}} = 0$. Therefore,
$$d_k = mn = C$$
This shows that after $k$ iterations, we have a specific value of the deck id $d_k$, and this value is the position of the desired card from the top of the deck. Thus, the trick $(C, n, n-1, k)$ is always solvable for $k > \log_n(C-1)$, and $(C, n, n-1, k) = C$.


\end{proof}

\begin{theorem} (Step (7)) \label{thm:step7}
Suppose $m = \dfrac{C}{n}$ and $b = \dfrac{mj}{n-1}$. If $C, n > 1$ and $0<j<n-1$, then $(C, n, j, k)$ is solvable if, and only if, $(n-1) \nmid mj$. In this case, $(C, n, j, k)$ is solvable for any integer $k > \log_nt$, where $t = \max\Big\{\dfrac{C-bn}{1-\{b\}}, \dfrac{bn-1}{\{b\}}\Big\}$ and, $(C, n , j, k) = mj + \floor{b} + 1$. 
\end{theorem}

\begin{proof}
Before we proceed further, we first verify that each of these logarithms are defined. As $0 < j < n-1$, hence,
\begin{align*}
mj & < m(n-1) \\
b(n-1) & < m(n-1) \\
b &< m \qquad \qquad \text{ as } \hspace{0.1in} n > 1 \\
bn &< mn \\
C &> bn \qquad \qquad \text{ as } \hspace{0.1in} C = mn
\end{align*}
Also, when the trick is solvable then $b \not\in \mZ$. So, $0 < \{b\} < 1$, hence $1 - \{b\} > 0$. Thus $\dfrac{C-bn}{1-\{b\}} > 0$, and hence, $\log_n\Big(\dfrac{C-bn}{1-\{b\}} \Big)$ is defined. For the other logarithm, we know that $C \geq n$, thus, $C > n-1$, and since $0 < j < n-1$, hence,
\begin{align*}
Cj &> n-1 \\
\frac{Cj}{n-1} &> 1 \\
\frac{mnj}{n-1} &>1 \\
bn & > 1 \qquad \qquad \text{ as } \hspace{0.1in} \frac{mj}{n-1} = b
\end{align*}
As $\{b\} > 0$, hence, $\dfrac{bn-1}{\{b\}} > 0$, and thus, $\log_n\Big(\dfrac{bn-1}{\{b\}}\Big)$ is defined.

\noindent
Now we come back to the main proof. 
\begin{center}
\underline{\textbf{Part 1}}
\end{center}
We first show that if $(n-1) \nmid mj$ then $(C, n, j, k)$ is solvable. As $(n-1) \nmid mj$, thus $b \not\in \mZ$. Hence $0<\{b\} < 1$. From Theorem \ref{thm6}, we have,
\begin{align*}
d_k = mj + \ceil[\Bigg]{\frac{b(n-1)n\Big(\dfrac{n^{k-1}-1}{n-1}\Big) + d_0}{n^k}} & = mj + \ceil[\Bigg]{\frac{b(n^k-n)}{n^k} + \frac{d_0}{n^k}} \\
& = mj + \ceil[\Bigg]{b + \frac{d_0 - bn}{n^k}}
\end{align*}
As $b = \floor{b} + \{b\}$ and $\floor{b} \in \mZ$, hence, by Lemma \ref{lem2},
$$d_k = mj + \floor{b} + \ceil[\Bigg]{\{b\} + \frac{d_0 - bn}{n^k}}$$
As usual, $1 \leq d_0 \leq C$. Thus, 
\begin{alignat*}{2}
\{b\} + \frac{1 - bn}{n^k} &\leq &\{b\} + \frac{d_0 - bn}{n^k} \, \, \, &\leq \{b\} + \frac{C - bn}{n^k} \numberthis \label{step7:ineq1}\\
\text{Hence, } \qquad \ceil[\Bigg]{\{b\} + \frac{1 - bn}{n^k}} &\leq &\ceil[\Bigg]{\{b\} + \frac{d_0 - bn}{n^k}} &\leq \ceil[\Bigg]{\{b\} + \frac{C - bn}{n^k}} \numberthis \label{step7:ineq2}
\end{alignat*}
Now suppose $t = \dfrac{C-bn}{1-\{b\}}$. Thus $k > \log_n\Big(\dfrac{C-bn}{1-\{b\}} \Big) \geq \log_n\Big(\dfrac{bn-1}{\{b\}}\Big)$. This implies,
\begin{align*}
n^k &> \frac{C-bn}{1-\{b\}} \qquad \text{ and } \qquad n^k > \frac{bn-1}{\{b\}}\\
\frac{C-bn}{n^k} &<1-\{b\} \qquad \text{ and } \qquad \frac{1-bn}{n^k} > -\{b\}\\
\{b\} + \frac{C-bn}{n^k} &<1 \qquad \qquad \text{ and } \qquad \{b\} + \frac{1-bn}{n^k} > 0\\
\end{align*}
We have already seen that $C>bn$ and $\{b\} >0$. Thus,
$$0 < \{b\} + \frac{C-bn}{n^k} < 1$$
This implies,
\begin{align*}
\ceil[\Bigg]{\{b\} + \frac{C - bn}{n^k}} = 1 \numberthis \label{step7:ineq3}
\end{align*}
From inequality \eqref{step7:ineq1}, we see that, 
$$\{b\} + \frac{1 - bn}{n^k} \leq \{b\} + \frac{C - bn}{n^k} < 1$$
Thus,
$$0 < \{b\} + \frac{1-bn}{n^k} < 1$$ 
This again implies,
\begin{align*}
\ceil[\Bigg]{\{b\} + \frac{1 - bn}{n^k}} = 1 \numberthis \label{step7:ineq4}
\end{align*}
Using results from \eqref{step7:ineq3} and \eqref{step7:ineq4} in inequality \eqref{step7:ineq2}, we have
$$1 \leq \ceil[\Bigg]{\{b\} + \frac{d_0 - bn}{n^k}} \leq 1$$
Therefore,
$$ \ceil[\Bigg]{\{b\} + \frac{d_0 - bn}{n^k}} = 1$$
and $d_k = mj + \floor{b} + 1$. As there is a specific value of $d_k$ after $k$ iterations, thus $(C, n, j, k)$ is solvable in this case with $(C, n, j, k) = mj + \floor{b} + 1$. 

\enspace
\noindent
Now suppose $t = \dfrac{bn-1}{\{b\}}$. Thus $k > \log_n\Big(\dfrac{bn-1}{\{b\}}\Big) \geq \log_n\Big(\dfrac{C-bn}{1-\{b\}} \Big)$. We follow the same argument as before to conclude that 
$$ \ceil[\Bigg]{\{b\} + \frac{d_0 - bn}{n^k}} = 1$$
and $d_k = mj + \floor{b} + 1$. Again, as there is a specific value of $d_k$ after $k$ iterations, thus $(C, n, j, k)$ is solvable in this case with $(C, n, j, k) = mj + \floor{b} + 1$. 


\noindent 
Now we prove the other direction.
\begin{center}
\underline{\textbf{Part 2}}
\end{center}
We have to prove that for $C, n>1$ and $0<j<n-1$, if $(C, n, j, k)$ is solvable then $(n-1) \nmid mj$. We will however prove the contrapositive of this statement as they are equivalent. Hence we will show that for $C, n>1$ and $0<j<n-1$, if $(n-1) \mid mj$, then $(C, n, j, k)$ is not solvable.

\noindent
As $(n-1) \mid mj$, thus $b \in \mZ$. Using Theorem \ref{thm6},
\begin{align*}
d_k = mj + \ceil[\Bigg]{\frac{b(n-1)n\Big(\dfrac{n^{k-1}-1}{n-1}\Big) + d_0}{n^k}} & = mj + \ceil[\Bigg]{\frac{b(n^k-n)}{n^k} + \frac{d_0}{n^k}} \\
& = mj + \ceil[\Bigg]{b + \frac{d_0 - bn}{n^k}}\\
& = mj + b + \ceil[\Bigg]{\frac{d_0 - bn}{n^k}} \text{ as } b \in \mZ
\end{align*}
We know that $1 \leq d_0 \leq C$. Thus,
\begin{alignat*}{2}
\frac{1 - bn}{n^k} &\leq &\frac{d_0 - bn}{n^k} \, \, \, &\leq \frac{C - bn}{n^k} \numberthis \label{step7:ineq5}\\
\text{Hence, } \qquad \ceil[\Bigg]{\frac{1 - bn}{n^k}} &\leq &\ceil[\Bigg]{\frac{d_0 - bn}{n^k}} &\leq \ceil[\Bigg]{\frac{C - bn}{n^k}} \numberthis \label{step7:ineq6}
\end{alignat*}
We have already seen that $C > bn$. Thus,
\begin{align*}
\frac{C-bn}{n^k} &> 0 \\
\ceil[\Bigg]{\frac{C-bn}{n^k}} & > 0
\end{align*}
Similarly, we have also seen that $bn > 1$. Hence,
\begin{align*}
\frac{1-bn}{n^k} &< 0 \\
\ceil[\Bigg]{\frac{1-bn}{n^k}} & \leq 0
\end{align*}
As $\ceil[\Bigg]{\dfrac{C-bn}{n^k}} \in \{1, 2, 3, \ldots\}$ and $\ceil[\Bigg]{\dfrac{1-bn}{n^k}} \in \{0, -1, -2, \ldots \}$, hence from inequality \eqref{step7:ineq6}, $\ceil[\Bigg]{\dfrac{d_0 - bn}{n^k}}$ does not yield any specific integer for any $k$. Thus, $d_k = mj + b + \ceil[\Bigg]{\dfrac{d_0 - bn}{n^k}}$ is also not a specific integer for any $k$. Therefore, $(C, n, j, k)$ is not solvable.

\end{proof}

Note that from theorems \ref{thm:step3}, \ref{thm:step5}, \ref{thm:step6}, \ref{thm:step7}, we see that the value of a solvable trick $(C, n, j, k)$ does not depend on $k$. Hence, two solvable tricks $(C_1, n_1, j_1, k_1)$ and $(C_2, n_2, j_2, k_2)$ will be considered the \textbf{same trick} if $C_1 = C_2, n_1 = n_2,$ and $j_1 = j_2$.

\section{Other Interesting Results}

\noindent
Now that we have a complete mathematical understanding of the trick, we ask ourselves some interesting questions. 
\begin{enumerate}
\item Given a number of cards $C$, how how many solvable magic tricks are there?
\item How does our algorithm solve the 21CT?
\item Assuming $C, n > 1$, what choices of $C$ and $n$ will guarantee that $(C, n, j, k)$ is solvable for all $0 \leq j \leq n-1$ and an appropriate $k$.
\item Assuming $C, n > 1$, what choices of $C$ and $n$ will guarantee that $(C, n, j, k)$ is not solvable for any $0 < j < n-1$.
\end{enumerate}
We answer these questions in the form of a theorem and several corollaries. 
\begin{theorem}\label{thm:interestingresult1}
For a given number of cards $C \geq 1$, the number of solvable magic tricks $p$, is given by 
$$p =
\begin{cases}
1, & C = 1 \\
\sum\limits_{\substack{n>1 \\ n \mid C }} \Big(n+1 - \gcd\Big(\frac{C}{n}, n-1\Big)\Big) & C>1, n>1
\end{cases}
$$
\end{theorem}

\begin{proof}
\noindent
\underline{Case 1  ($C=1$):} If $C= 1$, then $n = 1$, therefore, there is only one trick $(1, 1, 0, k)$ which is solvable for any $k >0$ by Theorem \ref{thm:step3}. Thus $p=1$.

\enspace
\noindent
\underline{Case 2  ($C>1$):} If $n = 1$, then there is again only one trick $(C, 1, 0, k)$ which is not solvable for any $k>0$ by Theorem \ref{thm:step4}. Therefore, $p=0$. 
Now suppose $C>1$ and $n>1$. We choose a fixed $n$ that is a divisor of $C$. We will first find out how many tricks are not solvable. By Theorem \ref{thm:step7}, if $0 < j < n-1$, then $(C, n, j, k)$ is not solvable for any $k$, if $(n-1) \mid mj$ where $m = \dfrac{C}{n}$. As $0 < j < n-1$, thus $(n-1) \nmid j$. The possible values of $mj$ are $m, 2m, 3m, \ldots, (n-2)m$. There are $n-2$ such values. We thus need to find out which multiples of $m$ are also multiples of $n-1$ in order for the trick $(C, n, j, k)$ to be not solvable. The answer happens to be the multiples of $\lcm(m, n-1)$ from the definition of the least common multiple. However, we need to find out how many multiples of $\lcm(m, n-1)$ exist in the set $\{m, 2m, 3m, \ldots, (n-2)m\}$. This is given by $\dfrac{m(n-1)}{\lcm(m, n-1)} - 1$. We subtract 1 as $m(n-1)$ is one such multiple that does not belong in the set. From elementary number theory \cite{Burton}, we see that $\dfrac{m(n-1)}{\lcm(m, n-1)} = \gcd(m, n-1)$. Thus, the number of tricks $(C, n, j, k)$ that are not solvable for $0 < j < n-1$ is $\gcd(m, n-1) - 1$. So, the number of tricks that are solvable for $0 < j < n-1$ is $(n-2) - (\gcd(m, n-1) - 1) = n-1 - \gcd(m, n-1)$. Now for $j = 0$ and $j=n-1$ we have seen that $(C, n, j, k)$ is solvable by Theorems \ref{thm:step5} and \ref{thm:step6}. These add 2 more solvable tricks. Hence, the total number of solvable tricks $(C, n, j, k)$ for a fixed $n$, and an appropriate $k$ is 
$$n-1 - \gcd(m, n-1) + 2 = n+1 -\gcd\Big(\frac{C}{n}, n-1\Big)$$
As $n > 1$ cycles through the divisors of $C$, we can see the total number of solvable tricks for a particular $C$ to be 
$$p = \sum\limits_{\substack{n>1 \\ n \mid C }} \Big(n+1 - \gcd\Big(\frac{C}{n}, n-1\Big)\Big)$$
\end{proof}

\begin{corollary}
The 21CT, $(21, 3, 1, 3)$ is solvable.
\end{corollary}

\begin{proof}
We use Step (7) of our algorithm as $C, n >1$ and $0 < j < 2$. Here $m = 21/3 = 7$, $mj = 7\cdot 1 = 7$, and $n-1 = 2$. As, $ 2 \nmid 7$, hence $(21, 3, 1, k)$ is solvable for any integer $k > \log_nt$. We now find $t$. In order to find $t$, we need to know $b$ and $\{b\}$. $b = \dfrac{mj}{n-1} = \dfrac{7}{2}$. Hence, $\floor{b} = 3$, and $\{b\} = \dfrac{1}{2}$.

\noindent
$t = \max \Bigg\{ \dfrac{21-\Big(\dfrac{7}{2}\Big)\cdot 3}{1-\dfrac{1}{2}}, 
 \dfrac{\Big(\dfrac{7}{2}\Big)\cdot 3 - 1}{\dfrac{1}{2}}\Bigg\} = \max \{21, 19\} = 21$. As $k > \log_3 21 = 2.771$, hence, $k = 3$. Thus, the 21CT, $(21, 3, 1, 3)$ is solvable. Moreover, $d_3 = mj + \floor{b} + 1 = 7 + 3 + 1 = 11$. Therefore, $(21, 3, 1, 3) = 11$.
\end{proof}

\begin{corollary}\label{cor:intresults3}
Suppose $C, n > 1$. If $\gcd\Big((n-1), \dfrac{C}{n}\Big) = 1$, then $(C, n, j, k)$ is solvable for all $0 \leq j \leq n-1$ and an appropriate $k$. 
\end{corollary} 

\begin{proof}
By Theorem \ref{thm:interestingresult1}, the number of solvable tricks for a specific $n$, is equal to $n+1-\gcd\Big(\dfrac{C}{n}, n-1\Big)$. As $\gcd\Big(\dfrac{C}{n}, n-1\Big) = 1$, hence the number of solvable tricks is $n+1 - 1 = n$ which are for all the possible values of $j = 0, 1, 2, \ldots, n-1$. Hence, $(C, n, j, k)$ is solvable for all $0 \leq j \leq n-1$ and an appropriate $k$.

\end{proof}

\begin{corollary}
Suppose $C, n > 1$. If $C = n(n-1)$, then $(C, n, j, k)$ is not solvable for any $0 < j < n-1$. To find 
\end{corollary}

\begin{proof}
Using previous notation, $m = \dfrac{C}{n}$, we see that $C = n(n-1)$ is the same as $m = n-1$. Thus, $n-1 \mid mj$ for all $0 < j < n-1$. Hence, by Theorem \ref{thm:step7}, $(C, n, j, k)$ is not solvable.
\end{proof}

\section{Conclusion}

We have found a complete answer to the \hyperref[gen-ques]{general question} posed at the start of Section 4 (Main Result). The \hyperref[algorithm]{algorithm} presented in that section shows us a way of determining if a trick is solvable while starting with $C$ cards, split into $n$ stacks, with $j$ stacks going on top of the stack with the desired card after every iteration. In brief, here is a summary of questions that are answered in this paper.
\begin{itemize}
\item If Magi is handed $C$ $(>1)$ cards and is asked to split them into $n$ $(>1)$ stacks,  without any specific $j$, then he can perform the trick $(C, n, 0, k)$ or $(C, n, n-1, k)$ for appropriate $k$'s as they are solvable due to Theorem \ref{thm:step5}, and \ref{thm:step6} respectively.
\item If Magi is handed $C$ $(>1)$ cards and is asked to split them into $n$ $(>1)$ stacks, with a specific $0<j <n-1$, then he can use Theorem \ref{thm:step7} to determine if the trick $(C, n, j, k)$ is solvable for any $k$ and proceed. 
\end{itemize}
In addition, we also found the total number of solvable tricks for a given $C$ in Theorem \ref{thm:interestingresult1}. Finally, using our algorithm we list some solvable tricks other than the 21CT below, that can be performed by anyone with a deck of cards  to impress their friends. 

\smallskip
\noindent
\begin{table}[H]
\centering
\begin{tabular}{|c|c|c|}
\hline 
$(20, 4, 2, 3) = 14$ & $(28, 4, 2, 3) = 19$ & $(36, 6, 4, 3) = 29$ \\ 
\hline 
$(21, 7, 5, 2) = 18$ & $(30, 5, 3, 3) = 23$ & $(36, 9, 3, 2) = 14$ \\ 
\hline 
$(24, 6, 4, 3) = 20$ & $(32, 4, 2, 3) = 22$ & $(39, 3, 1, 4) = 20$ \\ 
\hline 
$(25, 5, 3, 3) = 19$ & $(33, 3, 1, 4) = 17$ & $(40, 4, 2, 3) = 27$ \\ 
\hline 
$(27, 3, 1 , 4) = 14$ & $(35, 5, 3, 3) = 27$ & $(40, 8, 5, 2) = 29$ \\ 
\hline 
\end{tabular} 
\caption{List of 15 solvable tricks}
\end{table}

\end{document}